\documentclass{article}

\usepackage[margin=2.5cm]{geometry} 

\usepackage[nolist]{acronym}

\usepackage{graphicx}
\usepackage[dvipsnames]{xcolor}
\usepackage{soul}
\usepackage{url}
\usepackage{subcaption}
\usepackage{algorithm2e}
\RestyleAlgo{ruled}
\makeatletter
\let\NAT@parse\undefined
\makeatother
\usepackage{hyperref}
\usepackage{tikz}
\usepackage{bbm}
\usepackage{epsfig}
 
\usepackage{amssymb}
\usepackage{amsthm}
\usepackage{cancel, comment, nicefrac}
\usepackage{mathtools}
\usepackage{comment}

\newtheorem{lemma}{Lemma}
\newtheorem{proposition}{Proposition}
\theoremstyle{definition}
\newtheorem{assumption}{Assumption}
\newtheorem{definition}{Definition}
\newtheorem{remark}{Remark}
\usepackage{multicol}

\usepackage{booktabs}

\definecolor{matt}{rgb}{0.5 0 1}
\setstcolor{matt}

\definecolor{mg}{rgb}{0 0 1}
\setstcolor{mg}

\newcommand{\cA}{\mathcal{A}}

\newcommand{\cD}{\mathcal{D}}

\newcommand{\cF}{\mathcal{F}}

\newcommand{\cL}{\mathcal{L}}

\newcommand{\cR}{\mathcal{R}}

\newcommand{\cX}{\mathcal{X}}

\newcommand{\N}{\mathbb{N}}

\newcommand{\R}{\mathbb{R}}

\newcommand{\bF}{\mathbf{F}}

\newcommand{\bQ}{\mathbf{Q}}

\begin{document}

\begin{acronym}
\acro{RoA}{region of attraction}
\acro{PWA}{piecewise affine}
\acro{LP}{linear program}
\end{acronym}

\title{Data-driven approximation of regions of attraction via an LP-based selection of PWA Lyapunov functions}

\author{Oumayma Khattabi, Matteo Tacchi-Bénard\thanks{The work of Matteo Tacchi-Bénard is supported by the French National Research Agency (ANR) under the PIVOINE project grant ANR-25-CE48-1599-01.}, Martin Gulan\thanks{The work of Martin Gulan is supported by the Scientific Grant Agency of the Slovak Republic under the grant 1/0401/26.}, Sorin Olaru}

\maketitle

\begin{abstract}
This paper presents a method to approximate regions of attraction of unknown nonlinear dynamical systems from data. Assuming point-wise evaluations of the vector field and known Lipschitz bounds, a polyhedral uncertainty set of admissible dynamics is constructed. This uncertainty description enables the synthesis of a continuous \ac{PWA} Lyapunov candidate via a \ac{LP}, enforcing a robust decrease condition for all admissible vector fields. The approach allows certification of a  \ac{RoA} consistent with the available data. Numerical examples illustrate the effectiveness of the proposed method in extracting certified \ac{RoA}s from sparse data.

\end{abstract}

\section{Introduction}
In control theory, stability analysis is the theoretical foundation for system analysis and the construction of stability certificates. It is typically carried out using a mathematical model of the system dynamics, which is often unavailable or difficult to obtain for nonlinear systems. This limitation can be mitigated through data-driven methods, which have seen a significant increase in use as data and computational power have become widely available. Data-driven stability analysis is carried out through different approaches, among which the construction of Lyapunov functions~\cite{lyapunov1992} or the verification of barrier certificates~\cite{nagumo1942} are common. These rely on a range of computational tools, including neural networks and convex optimization frameworks, but certifying stability properties such as regions of attraction from data remains challenging.

Neural network-based methods have attracted increasing attention due to their learning and approximation capabilities. For instance, in~\cite{kolter2019}, a Lyapunov candidate is learned jointly with the dynamics to certify stability, whereas in~\cite {wang2024} it is learned jointly with the controller to ensure convergence. In~\cite{grune2021}, the approximation of a Lyapunov candidate is used to counteract the curse of dimensionality typically encountered in neural network-based approaches. Neural networks can also be used to learn invariant sets via learned Lyapunov functions for general reference-tracking control laws~\cite{kim2024}. Despite their flexibility, these approaches often lack explicit robustness guarantees and rely on sampled data, which can limit their ability to provide reliable certification.

Other approaches rely on convex optimization and polynomial methods. Sum-of-Squares programming is used in~\cite{martin2024} to compute polynomial representations of the system and certify stability through dissipativity constraints, and in~\cite{miller2024} to develop super-stabilizing controllers to counter measurement corruption and noise. Polynomial optimization is also used in~\cite{bramburger2024}, where, coupled with Koopman operators, dynamical system properties are certified, and in~\cite{zheng2024}, where a data-driven safe controller is developed with barrier constraints to guarantee stability and safety. Such methods rely on function approximations and structural assumptions, which may limit their applicability in fully data-driven settings.

Some early model-based works related to this paper include~\cite{brayton1979,brayton1980}, which proposed systematic procedures for constructing a Lyapunov function for nonlinear systems, and \cite{ohta1993}, which introduced computer-aided methods for the automatic generation of Lyapunov candidates. Other works, such as~\cite{kiendl1992,hmamed1994,loskot1998}, rely on vector norms as Lyapunov functions to derive stability conditions for linear systems, while stability margin evaluation for uncertain linear systems is studied in~\cite{gong1994,su1994}, highlighting robustness under bounded uncertainties. The above approaches provide strong theoretical guarantees but rely on explicit system models.

This paper uses an uncertainty set derived from data to induce a differential inclusion framework and construct a \ac{PWA} Lyapunov candidate. Related approaches using differential inclusions include~\cite{kamalapurkar2020}, which employs a set-valued map of the model for stability certification, and~\cite{dacs2023}, which develops a data-driven control barrier function using uncertainty estimation. Works constructing \ac{PWA} Lyapunov candidates include~\cite{julian1999}, which proposes an \ac{LP} that builds \ac{PWA} Lyapunov function over a polyhedral partition for a known nonlinear system, and~\cite{tacchi2025}, which formulates a second-order cone program for bounded unknown system with data. While approaches based on differential inclusions and polyhedral uncertainty bounds, as well as approaches constructing Lyapunov functions from data sets, have been explored in the literature, the proposed approach combines explicit uncertainty sets derived from data with \ac{PWA} Lyapunov functions in a fully data-driven setting, resulting in a tractable \ac{LP}-based formulation for stability certification.

The remainder of this paper is organized as follows: problem formulation and assumptions are presented in Section~\ref{sec:preliminaries}, followed by formulation of the uncertainty set in Section~\ref{sec:uncertainty_set}. Section~\ref{sec:PWA_Lyapunov} introduces the Lyapunov candidate, which, together with the uncertainty set, define the \ac{LP} in Section~\ref{sec:algorithm}. Numerical results illustrating certified \ac{RoA}s are given in Section~\ref{sec:numerical_examples}, followed by conclusions in Section~\ref{sec:conclusion}.

\vspace{0.4cm}
\noindent\textbf{Notation:}
Let $n\in\N$, $[n] \triangleq \{1,\cdots,n\}$. 
$I_n$ denotes the $n\times n$ identity matrix, and $\mathbf{1}_n$ the vector of ones in $\R^n$. 
For any $k\in[n]$, $e_k$ denotes the $k$-th canonical basis vector in $\R^n$. 
$\R_{>0}^n$ is the set of $n$-dimensional vectors of positive real components.
The symbol $\cdot$ denotes matrix product, while $\times$ denotes the Cartesian product of sets.
Let $\cX$ be a set in $\R^n$; $\text{int}(\cX)$ denotes its interior.
If $\cX$ is a polyhedron, then $|\cX|$ denotes the number of its vertices.
The symbol $\nabla$ denotes the gradient operator.
$|\cdot|$ denotes the componentwise absolute value, and $||\cdot||_p$ is the vector $p$-norm for $p\in\{1,2,\infty\}$.
$||\cdot||$ denotes the Euclidean norm $||\cdot||_2$.

\section{Preliminaries and problem formulation}\label{sec:preliminaries}

Throughout this paper, we consider the dynamical system
\begin{equation}
    \dot{x} = f(x), \label{eq:ODE}
\end{equation}
where $f: \R^n\longrightarrow\R^n$, $n\in\N$, is an unknown vector field.

We seek to approximate the region of attraction (\ac{RoA}) of a stable equilibrium of the system within a convex, compact (\textit{i.e.}, bounded and closed) polyhedral set $\cX\subset\R^n$.

\subsection{Assumptions}
The following assumptions are adopted to formulate the problem and develop the proposed approach:
\begin{assumption}\label{assumption:equilibrium}
The system admits a unique locally asymptotically stable equilibrium in $\cX$, which, without loss of generality, is taken to be the origin.
\end{assumption}
\begin{assumption}\label{assumption:lipschitz}
The function $f$ is Lipschitz continuous on $\cX$, and a componentwise upper bound $M\in\mathbb{R}^n_{>0}$ on its Lipschitz constant is known.  Specifically, Lipschitz continuity is interpreted with respect to the $\cL_\infty$-norm:
\begin{equation*}
    |f_\kappa(x)-f_\kappa(y)| \leq  M_\kappa ||x-y||_{\scriptscriptstyle \infty}, \quad \forall x,y\in\cX, \ \ \forall \kappa\in[n].
\end{equation*}
\end{assumption}

\begin{assumption}\label{assumption:dataset}
$f(\cdot)$ is accessible for point-wise evaluation, either through simulation or measurements. This implies the availability of a data set $\cD_{N_d} = \{(x_i,f_i)\}_{i\in[N_d]}$, where $x_i\in\cX$, $f_i = f(x_i)$ and $N_d\in\N$.
\end{assumption}

\begin{assumption}\label{assumption:neighbourhood_A}
An arbitrarily small polyhedral set $\cA \subset \cX$ containing the equilibrium is available, and the behavior of the system in this set is safe.
\end{assumption}

\subsection{Problem formulation}
This work aims to certify the system's stability in a region around the equilibrium using data. The proposed method consists of constructing a \ac{PWA} Lyapunov function and sub-sequently extracting a certified \ac{RoA} from its level set.

The approximation of the \ac{RoA} is carried out in three main steps. The first step consists of identifying the uncertainty set, \textit{i.e.}, the set of all admissible pairs $(x,f(x))$ over $\cX$ consistent with the available data and the Lipschitz bound. The second step consists of partitioning the state space into a tessellation and, using the uncertainty set, constructing a \ac{PWA} Lyapunov candidate. The third step is to extract a valid \ac{RoA} with stability certificates. The remainder of the paper details these three steps in the following sections.

\section{Characterizing uncertainty in the dynamics from data information}
\label{sec:uncertainty_set}

Based on Assumptions~\ref{assumption:lipschitz} and~\ref{assumption:dataset}, the uncertainty set consistent with the data is defined as:
\begin{equation}
\label{eq:uncertainty_set}
    F_{\cD_{N_d}} = \{ (x,z) \in \cX \times \R^n| \, \forall i \in [N_d], || z - f_i ||_{\scriptscriptstyle \infty} \leq  M|| x - x_i ||_{\scriptscriptstyle \infty} \}.
\end{equation}

This basic characterization can be refined to exploit locally active restrictions and obtain a global representation more suitable for analysis; a reformulation of the Lipschitz inequalities is in order.
Since the objective is to determine bounds on each component of $f(x)$, we focus on a generic component indexed by $\kappa$. At this stage, the Lipschitz constant can be specified componentwise. Hence, we associate to the component $f_\kappa(x)$ its Lipschitz constant $M_\kappa$:
\begin{equation}
    \left| f_\kappa(x) - f_{\kappa,i} \right| \leq M_\kappa \underset{k\in[n]}{\max} \left ( |x_k - x_{k,i}| \right ). \label{eq:K-ineq}
\end{equation}

Let $\kappa' \triangleq \underset{k\in[n]}{\arg\max} \left ( |x_k - x_{k,i}| \right )$, then the following inequalities hold:
\begin{gather*}
    \left| f_\kappa(x) - f_{\kappa,i} \right| \leq M_\kappa \left| x_{\kappa'} - x_{\kappa',i} \right|, \\
    \left| x_{k} - x_{k,i} \right| \leq | x_{\kappa'} - x_{\kappa',i} | \quad \forall k \in[n].
\end{gather*}

For each selection of indices $(\kappa, \kappa')$, two alternatives for the admissible uncertainty in the component $f_\kappa(x)$ arise naturally depending on the sign of $\left( x_{\kappa'} - x_{\kappa',i} \right)$:
\begin{gather}
    P^{(+)}_{\scriptscriptstyle i,\kappa,\kappa'} = \left\{ (x,z) \in \cX \times \R : \left | \begin{array}{c} x - x_i \\ z - f_{\kappa,i} \end{array} \right | \leq     \begin{bmatrix} \mathbf{1}_n \\ M_\kappa \end{bmatrix} \left( x_{\kappa'} - x_{\kappa',i} \right) \right\} , \label{eq:pcone}\\
    P^{(-)}_{\scriptscriptstyle i,\kappa,\kappa'} = \left\{ (x,z) \in \cX \times \R : \left | \begin{array}{c} x - x_i \\ z - f_{\kappa,i} \end{array} \right | \leq - \begin{bmatrix} \mathbf{1}_n \\ M_\kappa \end{bmatrix} \left( x_{\kappa'} - x_{\kappa',i} \right) \right\}. \label{eq:ncone}
\end{gather}

The uncertainty in each element of the right-hand side of the ordinary differential equation~\eqref{eq:ODE} is represented by a scalar variable $z \triangleq f_\kappa(x) $; the admissible pairs $(x,z)$ are implicitly characterized through the inequality constraints.

Each of the sets~\eqref{eq:pcone} and~\eqref{eq:ncone} defines the inequalities of a polyhedral cone in $\R^{n+1}$ induced by the $\cL_\infty$-norm Lipschitz constraint, and will henceforth be referred to simply as a cone. Moreover, these cones are symmetric with respect to the axis $x_{\kappa'}$, sharing the same vertex and direction. A formal definition of this set is a double cone\footnote{Two cones of the same vertex and direction, and an opposing sense.}, of vertex $(x_i,f_{\kappa, i})$, of axis along $x_{\kappa'}$, and of opening angle in $\cL_\infty$-norm $\eta_\kappa\in \R^{n+1}$:
\begin{equation*}
    \eta_\kappa \triangleq \begin{bmatrix} \mathbf{1}_n \\ M_\kappa \end{bmatrix}.
\end{equation*}

\begin{remark}
    The projection of each cone is the same for all elements of $f(\cdot)$. This projection is obtained by deleting the $\{n+1\}$-th inequality for both~\eqref{eq:pcone} and~\eqref{eq:ncone}, which is the only inequality dependent on $f_\kappa(\cdot)$ and $M_\kappa$.
\end{remark}

A polyhedral representation of the cones~\eqref{eq:pcone} and~\eqref{eq:ncone} is given by:
\begin{align}
    P^{(+)}_{i,\kappa,\kappa'} = \left\{ (x,z)\in\cX\times\R \mid A^{+} \cdot \begin{bmatrix} x \\ z \end{bmatrix} \leq b^{+} \right\},
    \label{eq:pcone_poly}\\
    P^{(-)}_{i,\kappa,\kappa'} = \left\{ (x,z)\in\cX\times\R \mid A^{-} \cdot \begin{bmatrix} x \\ z \end{bmatrix} \leq b^{-} \right\},
\end{align}
with
\begin{equation*}
    A^{+} = \begin{bmatrix}
        I_{n+1} \\ \hline -I_{n+1}
    \end{bmatrix} -
    \begin{bmatrix}
        \eta_\kappa \cdot e_{\kappa'}^\mathsf{T} \\
        \hline \eta_\kappa \cdot e_{\kappa'}^\mathsf{T}
    \end{bmatrix}, \,
    A^{-} = \begin{bmatrix}
        I_{n+1} \\ \hline -I_{n+1}
    \end{bmatrix} +
    \begin{bmatrix}
        \eta_\kappa \cdot e_{\kappa'}^\mathsf{T} \\
        \hline \eta_\kappa \cdot e_{\kappa'}^\mathsf{T}
    \end{bmatrix},
\end{equation*}
\begin{equation*}
    b^{+} = \begin{bmatrix}
        x_{1,i} - x_{\kappa',i} \\
        \vdots \\
        x_{n,i} - x_{\kappa',i} \\
        f_{\kappa,i} - M_\kappa x_{\kappa',i} \\
        \hline -x_{1,i}-x_{\kappa',i} \\
        \vdots \\
        -x_{n,i} - x_{\kappa',i} \\
        -f_{\kappa,i} - M_\kappa x_{\kappa',i}
    \end{bmatrix}, \
    b^{-} = \begin{bmatrix}
        x_{1,i} + x_{\kappa',i} \\
        \vdots \\
        x_{n,i} + x_{\kappa',i} \\
        f_{\kappa,i} + M_\kappa x_{\kappa',i} \\
        \hline -x_{1,i} + x_{\kappa',i} \\
        \vdots \\
        -x_{n,i} + x_{\kappa',i} \\
        -f_{\kappa,i} + M_\kappa x_{\kappa',i},
    \end{bmatrix},
\end{equation*}
where $e_{\kappa'} = [0, \cdots,\underset{\kappa' \textnormal{-th position}}{1}, 0, \cdots, 0]^\mathsf{T}\in \R^{n+1}$.

At this stage, the uncertainty sets of $f_\kappa(x)$ induced by the data set $\cD_{N_d}$ have been characterized by considering a single maximizer in the right-hand side of inequality~\eqref{eq:K-ineq}. To obtain the set of all admissible values of $f_\kappa(x)$, \textit{i.e.}, the uncertainty set of $f_\kappa(x)$ over $\cX$ given $(x_i,f_i)$, one must consider all indices attaining the maximum in the right-hand side of~\eqref{eq:K-ineq}. This set is given by the union of $2n$ disjoint cones:
\begin{equation}
    \mathbf{P}_{i,\kappa} = \bigcup_{l\in [n]} P^{(+)}_{i,\kappa,l} \cup P^{(-)}_{i,\kappa,l}. \label{eq:cone_multipoly}
\end{equation}

Figure~\ref{fig:multipoly} illustrates the union of polyhedra $\mathbf{P}_{i,\kappa}$ defined in~\eqref{eq:cone_multipoly} for $n=2$.
\begin{figure}[hbtp]
    \centering
    \includegraphics[width=0.6\linewidth]{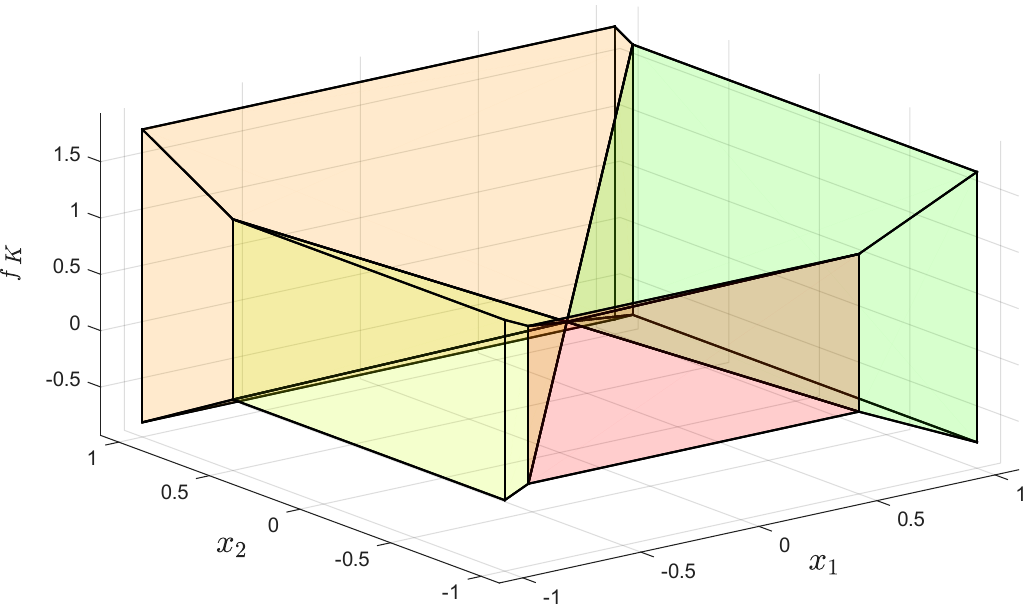}
    \caption{\centering\footnotesize Illustration of $\mathbf{P}_{i,\kappa}$ with $\kappa\in\{1,2\}$ (in green, yellow, orange and red are $P^{(+)}_{i,\kappa,1}$, $P^{(-)}_{i,\kappa,1}$, $P^{(+)}_{i,\kappa,2}$ and $P^{(-)}_{i,\kappa,2}$ respectively).}
    \label{fig:multipoly}
\end{figure}

Using the full data set $\cD_{N_d}$, the uncertainty set of $f_\kappa(x)$ over $\cX$ is obtained as:
\begin{equation}
    Q_\kappa = \bigcap_{i\in [N_d]} \mathbf{P}_{i,\kappa}. \label{eq:uncertainty_set_kappa}
\end{equation}

Figure~\ref{fig:uncertainty_set_K} illustrates $Q_\kappa$ for $n=2$.
\begin{figure}[hbtp]
    \centering
    \includegraphics[width=0.6\linewidth]{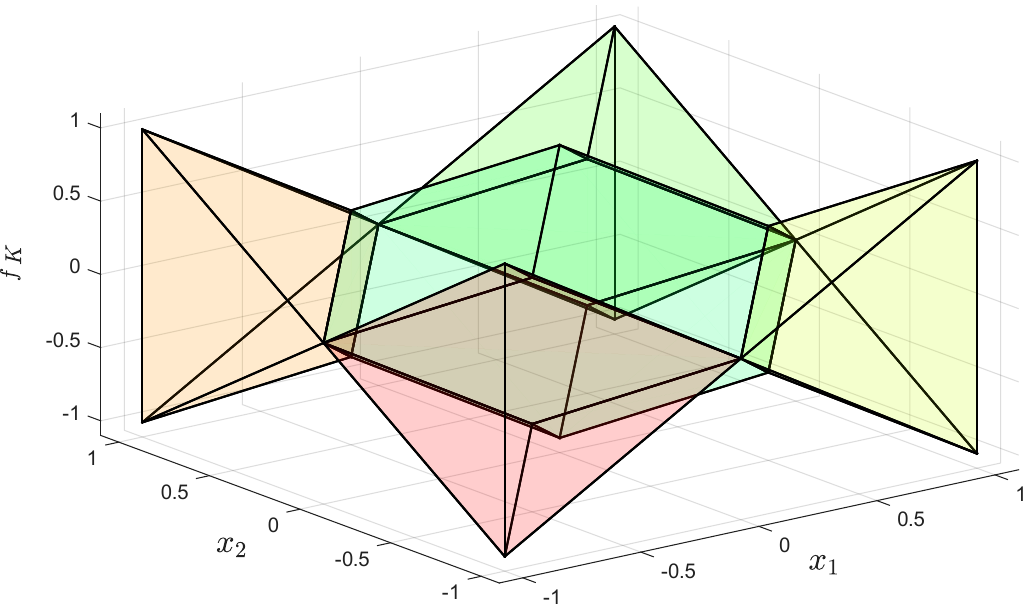}
    \caption{\footnotesize Illustration of $Q_\kappa$ with 4 data points.}
    \label{fig:uncertainty_set_K}
\end{figure}

We now analyze the geometric structure of the sets $Q_\kappa$ and their projections.

\begin{lemma}\label{lemma:Qk_union_polyhedra}
For any $\kappa \in [n]$, the set $Q_\kappa$ defined in \eqref{eq:uncertainty_set_kappa} is a finite union of polyhedral sets. Moreover, if $\cX$ is bounded, then $Q_\kappa$ is bounded.
\end{lemma}

\begin{proof}
For each $i\in [N_d]$ and $\kappa \in [n]$, the set $\mathbf{P}_{i,\kappa}$ is defined as a union of $2n$ sets of the form $P^{(+)}_{i,\kappa,l}$ and $P^{(-)}_{i,\kappa,l}$, with $l\in[n]$. By their polyhedral representation, each of these sets is a polyhedron in $\R^{n+1}$. Therefore, $\mathbf{P}_{i,\kappa}$ is a finite union of polyhedral sets.

Since $Q_\kappa$ is obtained as the intersection over $i\in[N_d]$ of the sets $\mathbf{P}_{i,\kappa}$, and each $\mathbf{P}_{i,\kappa}$ is a finite union of polyhedral sets, the distributivity of finite intersections over finite unions of sets implies that $Q_\kappa$ is itself a finite union of polyhedral sets.

If $\cX$ is bounded, then the $x$-component of every $(x,z)\in Q_\kappa$ is bounded. The inequalities defining the sets $P^{(+)}_{i,\kappa,l}$ and $P^{(-)}_{i,\kappa,l}$ then imply a bound on the scalar variable $z$, since
\begin{equation*}
|z-f_{\kappa,i}| \le M_\kappa \|x-x_i\|_{\scriptscriptstyle \infty}.
\end{equation*}
Hence $Q_\kappa$ is bounded.
\end{proof}

\begin{lemma}\label{lemma:projection_Qk}
For any $\kappa \in [n]$, the projections onto the state space of the regions associated with the construction of $Q_\kappa$ admit a common refinement that defines a finite polyhedral partition of $\cX$. Moreover, this partition is independent of $\kappa$.
\end{lemma}

\begin{proof}
For each $i\in[N_d]$, the sets $P^{(+)}_{i,\kappa,l}$ and $P^{(-)}_{i,\kappa,l}$ are obtained by enumerating the possible indices attaining the maximum in
\begin{equation*}
\|x-x_i\|_{\scriptscriptstyle \infty} = \max_{k\in[n]} |x_k-x_{k,i}|.
\end{equation*}
For each $i\in[N_d]$, this enumeration induces a finite collection of polyhedral regions $\{R_{i,l}\}_{l \in L_i}$ in $\mathbb{R}^n$ whose union is $\mathbb{R}^n$ and whose interiors are pairwise disjoint.

Consider the common refinement of these collections over all $i\in[N_d]$, \textit{i.e.}, the family of nonempty sets
\begin{equation*}
R_{l_1,\dots,l_{N_d}} = \cX\cap\bigcap_{i\in[N_d]} R_{i,l_i}, \quad l_i \in L_i.
\end{equation*}
Each such set is polyhedral, being the intersection of finitely many polyhedral sets. Moreover, the union of all nonempty sets $R_{l_1,\dots,l_{N_d}}$ is $\cX$, and their interiors are pairwise disjoint. Hence, this refinement defines a finite polyhedral partition of $\cX$. Since the inequalities constraining $x$ do not depend on $\kappa$, the resulting partition of $\cX$ is independent of $\kappa$.
\end{proof}

\begin{lemma}\label{lemma:Qkx}
For any $\kappa \in [n]$ and any $x \in \cX$, define
\begin{equation*}
\bQ_\kappa(x)\triangleq\{ z \in \R \mid (x,z)\in Q_\kappa \}.
\end{equation*}
Then $\bQ_\kappa(x)$ is a polyhedral set in $\mathbb{R}$; in particular, it is a bounded interval.
\end{lemma}

\begin{proof}
For any fixed $x\in\cX$ and any data point $(x_i,f_i)$, the Lipschitz inequality gives $|z-f_{\kappa,i}| \le M_\kappa \|x-x_i\|_{\scriptscriptstyle \infty}$, which is equivalent to
\begin{equation*}
f_{\kappa,i}-M_\kappa \|x-x_i\|_{\scriptscriptstyle \infty} \le z \le f_{\kappa,i}+M_\kappa \|x-x_i\|_{\scriptscriptstyle \infty}.
\end{equation*}
Thus, for fixed $x$, each data point defines a bounded interval of admissible values for $z$. Since $Q_\kappa$ is obtained by enforcing these inequalities for all $i\in[N_d]$, $\bQ_\kappa(x)$ is the intersection of finitely many bounded intervals. Therefore, $\bQ_\kappa(x)$ is itself a bounded interval in $\mathbb{R}$.
\end{proof}

As a consequence of Lemma~\ref{lemma:Qkx}, for any $\kappa \in [n]$ and any $x\in\cX$, there exist scalars
$f_\kappa^{\min}(x)$ and $f_\kappa^{\max}(x)$ such that
\begin{equation*}
    \bQ_\kappa(x)= [f_\kappa^{\min}(x),\, f_\kappa^{\max}(x)].
\end{equation*}
Therefore, the uncertainty induced by the data set can be described by the set-valued map
\begin{equation*}
\bF:\cX\rightrightarrows\mathbb{R}^n, \quad \bF(x)=\bQ_1(x)\times\cdots\times \bQ_n(x),
\end{equation*}
that is,
\begin{equation*}
\bF(x)= [f_1^{\min}(x),f_1^{\max}(x)] \times \cdots \times [f_n^{\min}(x),f_n^{\max}(x)].
\end{equation*}

From the componentwise characterization above, the global uncertainty set \eqref{eq:uncertainty_set} can equivalently be written as
\begin{equation}
\label{eq:uncertainty_set2}
F_{\mathcal{D}_{N_d}} = \{(x,z)\in\mathcal{X}\times\mathbb{R}^n \mid z \in \bF(x)\},
\end{equation}
\textit{i.e.}, the graph of all vector fields consistent with the data set
$\mathcal{D}_{N_d}$ over $\mathcal{X}$.

\begin{proposition}\label{prop:differential_inclusion}
Given the data set $\cD_{N_d}$ and Assumptions~\ref{assumption:lipschitz}--\ref{assumption:dataset}, the uncertainty in the vector field induced by the data set can be characterized by the differential inclusion
\begin{equation*}
\dot x \in \bF(x), \quad x\in\cX,
\end{equation*}
where $\bF:\cX\rightrightarrows\mathbb{R}^n$ is the set-valued map defined by $\bF(x)=\bQ_1(x)\times\cdots\times \bQ_n(x)$. Moreover, the graph of $\bF$ is represented by a finite union of polyhedral sets.
\end{proposition}

\begin{proof}
First, by Lemma \ref{lemma:Qk_union_polyhedra}, each set $Q_\kappa$ is a finite union of polyhedra in $\mathbb{R}^{n+1}$. By Lemma \ref{lemma:Qkx}, the set $\bQ_\kappa(x)$ is a bounded interval for every $x\in\cX$. Hence, for each $x\in\cX$, the set $\bF(x)=\bQ_1(x)\times\cdots\times \bQ_n(x)$ is a  hyperbox in $\mathbb{R}^n$ containing all admissible values of the vector field at $x$. Therefore, the uncertainty induced by the data set is described by the differential inclusion $\dot{x}\in\bF(x)$.

Moreover, by \eqref{eq:uncertainty_set2}, the graph of $\bF$ coincides with the global uncertainty set $F_{\cD_{N_d}}$, i.e. $\mathrm{graph}(\bF) = F_{\cD_{N_d}}$. Since each $Q_\kappa$ is a finite union of polyhedra and the projections onto the state space coincide for all $\kappa$ by Lemma \ref{lemma:projection_Qk}, the graph of $\bF$ is represented by a finite union of polyhedral sets.
\end{proof}

The set-valued map $\bF(x)$ provides, for each $x\in\cX$, a hyperbox in $\mathbb{R}^n$ containing all admissible values of the vector field consistent with the data. In view of subsequent stability analysis, it is convenient to exploit this polyhedral structure through its extreme points.

In particular, the hyperbox $\bF(x)$ is fully characterized by its set of vertices,
defined as follows:
\begin{equation}
\cF_v(x) = \left\{ z \in \R^n \mid z_k \in \{f_k^{\min}(x), f_k^{\max}(x)\}, \forall k \in [n] \right\}.
\end{equation}

\section{PWA Lyapunov candidate}\label{sec:PWA_Lyapunov}
This section presents the method used to derive a Lyapunov candidate function for the uncertainty set consistent with a data set $\cD_{N_d}$. When defining a Lyapunov candidate, the parametrization must satisfy at least two basic principles:
\begin{itemize}
    \item be a universal approximator of a generic nonlinear con-tinuous Lyapunov function;
    \item have manageable complexity in the number of degrees of freedom for the approximation.
\end{itemize}

While several alternatives exist, we focus here on \ac{PWA} functions:
\begin{definition}\label{def:PWA}
Given a polyhedral partition of $\cX\subset\R^n$, a function $g:\cX\rightarrow\R$ is called a \ac{PWA} function if
\begin{gather*}
    \cX = \bigcup_{j\in [N_c]} Y_j, \ \ \mathrm{int}(Y_j) \cap \mathrm{int}(Y_{j'}) = \varnothing, \,\forall j\neq j',\\
    g(x) =a_j^\mathsf{T}x + b_j, \ \ \forall x \in Y_j,
\end{gather*}
where $a_j\in\R^n$, $b_j\in\R$, and $N_c\in\N$.
\end{definition}

These functions are continuous and inherit the polyhedral geometry of the sets $Y_j$, while satisfying the two desiderata stated above. Moreover, \ac{PWA} functions involve a partition whose structure is closely related to the polyhedral partition induced by the differential inclusion describing the uncertainty in Section~\ref{sec:uncertainty_set}.

\subsection{State-space tessellation as a design parameter}

A \ac{PWA} function, according to Definition~\ref{def:PWA}, relies on a partition and associated coefficients $(a_j,b_j)$. In this work, we restrict attention to polyhedral tessellations of the state space, which provide a structured class of partitions suitable for our construction.
For Lyapunov-based certification, we treat the tessellation as a design degree of freedom and enforce Lyapunov-related properties by optimizing the coefficients $(a_j,b_j)$. 

Accordingly, we propose to construct the tessellation by selecting, in an initial phase of the procedure, its vertices $\{v_u\}_{u\in [N_v]}$. The rationale for this choice is the availability of points obtained as projections of the vertices of $\mathrm{graph}(\bF)$ onto the state space.
This choice serves two purposes. First, it adapts the tessellation complexity to the complexity of the dynamics using information encoded in the uncertainty set. Second, it ensures consistency between the state-space partition and the uncertainty set, which can be exploited in the certification procedure.

The tessellation is constructed from geometric features of $\mathrm{graph}(\bF)$. In particular, the vertices $\{v_u\}_{u\in [N_v]}$ are obtained by projecting the vertices of $\mathrm{graph}(\bF)$ onto the state space. This ensures that $\{v_u\}_{u\in [N_v]}$ capture the extreme points of the uncertainty set. Ideally, the resulting tessellation forms a subpartition of the partition induced by the projection of $\mathrm{graph}(\bF)$\footnote{In the numerical section, this structural constraint is relaxed.}.

\subsection{Lyapunov candidate}

Once the tessellation $\{Y_j\}_{j\in [N_c]}$ is imposed, the \ac{PWA} Lyapunov candidate $V:\cX\subset\R^n\rightarrow\R$ is defined as:
\begin{equation}\label{eq:Lyapunov}
    V(x) = a_j^\mathsf{T}x + b_j, \ \ \forall j\in [N_c], \ \forall x \in Y_j \subset \cX,
\end{equation}
where the coefficients are constrained to verify that:
\begin{itemize}
    \item $V$ is positive:
\begin{equation}\label{eq:positivity}
    V(x)>0,\; \forall x\in\cX\setminus\{0\},
\end{equation}
    \item $V$ is continuous:
\begin{equation}\label{eq:continuity}
    a_j^\mathsf{T}x + b_j = a_{j'}^\mathsf{T}x + b_{j'}, \; \forall j,j'\in [N_c],\; \forall x \in Y_j \cap Y_{j'},
\end{equation}
    \item and for each cell $Y_j$, the gradient of $V$ satisfies:
\begin{equation}\label{eq:decrease}
    \nabla V(x)^\mathsf{T} f(x) = a_j^\mathsf{T}f(x) < 0,\, \forall j\in [N_c],\, \forall x \in \mathrm{int}(Y_j)  \setminus  \{0\}.
\end{equation}
\end{itemize}

For $\alpha\in\R$, we define the sublevel set $\cL^V_\alpha$ of $V$ as:
\begin{equation*}
    \cL^V_\alpha \triangleq \{x \in \cX \mid V(x) \leq \alpha\}.
\end{equation*}

\begin{lemma}\label{lemma:levelset}
If there exists $V(\cdot)$ satisfying~\eqref{eq:positivity}--\eqref{eq:decrease} for all admissible $f(\cdot)\in \bF(\cdot)$, then, for any $\alpha\in\R$ with $\cL^V_\alpha \subseteq \cX$ and any $x_0\in \cL^V_\alpha$, either $x(t,x_0)\in \cL^V_\alpha$ for all $t\ge 0$, or there exists $t_e<\infty$ such that $x(t_e,x_0)\notin \cX$.
\end{lemma}

\begin{proof}
By Proposition~\ref{prop:differential_inclusion}, admissible dynamics satisfy $\dot x \in \bF(x)$. Since $V$ satisfies~\eqref{eq:positivity}--\eqref{eq:decrease} for all admissible $f(\cdot)\in \bF(\cdot)$, it is a Lyapunov function for this inclusion on $\cX$. Thus, for any $x_0\in \cL^V_\alpha$, as long as $x(t,x_0)\in \cX$, one has $V(x(t,x_0))\le V(x_0)\le \alpha$, and therefore $x(t,x_0)\in \cL^V_\alpha$. Hence, trajectories cannot leave $\cL^V_\alpha$ while remaining in $\cX$. Consequently, either $x(t,x_0)\in \cL^V_\alpha$ for all $t\ge 0$, or the trajectory exits $\cX$. In the latter case, continuity of $x(\cdot,x_0)$ implies the existence of a finite exit time $t_e<\infty$ such that $x(t_e,x_0)\notin \cX$.
\end{proof}

\begin{proposition}\label{prop:levelset}
Let $V(\cdot)$ satisfy~\eqref{eq:positivity}--\eqref{eq:decrease} for all admissible $f(\cdot)\in \bF(\cdot)$, and let $\cL^V_\alpha$ be a sublevel set for some $\alpha\in\R$, with $\cL^V_\alpha \subseteq \cX$. If $\cR \subset \cX$ is such that $\cR\cap \cL^V_\alpha$ is a non-empty connected set containing the equilibrium, then $\cR\cap \cL^V_\alpha$ defines a certified \ac{RoA} consistent with the data.
\end{proposition}

\begin{proof}
By Proposition~\ref{prop:differential_inclusion}, admissible dynamics satisfy $\dot x \in \bF(x)$, and $V$ is a Lyapunov function for this inclusion on $\cX$. By Lemma~\ref{lemma:levelset}, every trajectory starting in $\cL^V_\alpha$ remains in $\cL^V_\alpha$, since $\cL^V_\alpha \subseteq \cX$ excludes its escape in finite time from $\cX$. Hence, trajectories starting in $\cR\cap \cL^V_\alpha$ remain in $\cL^V_\alpha$. Since $\cR\cap \cL^V_\alpha$ is connected and contains the equilibrium, standard Lyapunov arguments imply convergence toward the equilibrium. Therefore, $\cR\cap \cL^V_\alpha$ defines a certified \ac{RoA}.
\end{proof}

\begin{proposition}\label{prop:gradV}
Let $x\in\cX\setminus\{0\}$. If
\begin{equation*}\label{eq:decrease*}
    \nabla V^\mathsf{T}z<0 \quad \forall z \in \cF_v(x) \tag{\ref{eq:decrease}*}
\end{equation*}
is verified, then~\eqref{eq:decrease} holds.
\end{proposition}

\begin{proof}
Since $\bF(x)$ is a convex polyhedral set (a hyperbox), any $f(x)\in \bF(x)$ can be expressed as a convex combination of its vertices $\cF_v(x)$. The mapping $z \mapsto \nabla V^\mathsf{T} z$ is linear, and thus its maximum over $\bF(x)$ is attained at the vertices, i.e.
\[ \max_{z \in \bF(x)} \nabla V^\mathsf{T} z = \max_{z \in \cF_v(x)} \nabla V^\mathsf{T} z. \]
If $\nabla V^\mathsf{T} z < 0$ for all $z \in \cF_v(x)$, then $\nabla V^\mathsf{T} z < 0$ for all $z \in \bF(x)$, and since $f(x)  \in  \bF(x)$, it follows that $\nabla V^\mathsf{T} f(x)  <  0$.
\end{proof}

The preceding results show that stability can be certified directly from data by constructing a Lyapunov function $V(\cdot)$ that decreases along all admissible dynamics consistent with the available information. In particular, Proposition~\ref{prop:levelset} shows that the connected component of the sublevel set $\cL^V_\alpha$ containing the equilibrium defines a certified \ac{RoA}.

To make this condition practically verifiable, the key idea behind Proposition~\ref{prop:gradV} is to replace the unknown vector field $f(\cdot)$ in~\eqref{eq:decrease} by a finite set of constraints evaluated at the vertices of the uncertainty set. Given the hyperbox structure of $\bF(x)$, these vertices fully characterize the admissible values of the vector field, so the decrease condition can be enforced through a finite number of inequalities.

\section{An algorithmic construction of certificates}\label{sec:algorithm}

In this section, the uncertainty set, the state-space tessellation, and the Lyapunov conditions derived previously are assembled into an optimization-based procedure. In particular, the finite characterization of the Lyapunov decrease condition provided by Proposition~\ref{prop:gradV} enables a tractable formulation. A feasible solution yields a valid \ac{PWA} Lyapunov candidate and an associated region of attraction. The resulting formulation admits an implementable algorithm proposed in this section.

\subsection{LP-based PWA candidate selection}\label{LP_based}

To construct the \ac{PWA} Lyapunov function~\eqref{eq:Lyapunov}, we formulate an optimization problem using the constraints~\eqref{eq:positivity}--\eqref{eq:decrease}, evaluated at the vertices $\{v_u\}_{u\in [N_v]}$ of the tessellation. The decrease condition~\eqref{eq:decrease} is replaced by its vertex-based counterpart~\eqref{eq:decrease*} using Proposition~\ref{prop:gradV}, by enforcing it at the vertices of the uncertainty set $\cF_v(v_u)$.

To allow for feasibility, decrease condition~\eqref{eq:decrease*} is relaxed using slack variables $\{s_u\}_{u\in[N_v]}$, lower-bounded by $-\mu$ with $\mu\in\R_{>0}$. The slack variables are minimized in the objective to enforce negativity of the Lyapunov decrease condition.

The resulting linear program (\ac{LP}) optimizes over variables $\{a_j\}_{j\in[N_c]}$, $\{b_j\}_{j\in[N_c]}$ and $\{s_u\}_{u\in[N_v]}$, and is formulated as:
\begin{subequations}\label{optimisation_problem}
\begin{align}
\min_{\{a_j,b_j,s_u\}} \;\; & \sum_{u} s_u \label{eq:cost_func} \\
\text{s.t.} \quad
s_u &\geq -\mu && \forall u, \label{cstr:negative_slack} \\
a_j^\mathsf{T} v_u + b_j &\geq 0 && \forall u,\,\forall j:\, v_u \in Y_j, \label{cstr:positivity} \\
a_j^\mathsf{T} v_u + b_j &= a_{j'}^\mathsf{T} v_u + b_{j'} && \forall u,\,\forall j,j': v_u \in Y_j \cap Y_{j'}, \label{cstr:continuity} \\
a_j^\mathsf{T} z &\leq s_u && \forall u,\,\forall j: v_u \in Y_j,\, \forall z \in \cF_v(v_u), \label{cstr:decrease}
\end{align}
\end{subequations}
where $u \in [N_v]$ indexes the tessellation vertices $\{v_u\}$ and $j,j' \in [N_c]$ index the cells $\{Y_j\}$.

For $\alpha\in\R_{>0}$, the connected component of the sublevel set $\cL_\alpha^V$ containing the origin, associated with the obtained Lyapunov function $V^*(x) = a_j^{*\mathsf{T}} x + b_j^*$, defines a certified region of attraction provided that all slack variables associated with vertices $v_u \in \cL_\alpha^V$ are strictly negative.

\subsection{Iterative procedure}

Before solving the optimization problem~\eqref{optimisation_problem}, several feasibility issues must be addressed.

A main obstacle to feasibility of constraint~\eqref{cstr:decrease} arises when the set $\bF(v_u)$ contains zero. Such a vector implies that $v_u$ is a potential equilibrium point, and the corresponding data set can be considered not sufficiently informative. In this case, \eqref{cstr:decrease} must hold for $z=0$:
$$a_j^\mathsf{T}0 \leq s_u \quad \Rightarrow \quad 0 \leq s_u,$$
meaning that the slack variable cannot be negative. Accordingly, the proposed solution is to add a data point at $v_u$ to reduce the uncertainty.

\begin{remark}
    The upper bound $M$ of the Lipschitz constant has great influence on the uncertainty set. It is feasible to estimate $M$ from a dataset~\cite{jin2020data,martin2024,makdesi2021efficient}, albeit out of the scope of the paper, but taking an estimated value equal to or lower than the real constant generates numerical issues and empty polyhedra. To stay on the safe side and ensure feasibility, it is better to scale $M$ up. However, the greater $M$ is than the real value, the more data points would be needed to counter the issue of the feasibility of constraint~\eqref{cstr:decrease}. Similarly, bounded noise can be incorporated in the uncertainty set formulation~\eqref{eq:uncertainty_set} by imposing $|\Tilde{x}_i - x_i|\leq \omega$ and $|\Tilde{f}_i - f_i|\leq \phi$, with $\Tilde{x}$ and $\Tilde{f}$ the measured evaluations, and $\omega$ and $\phi$ the bounds on the noise. It would induce the same need for more data points to better cover the set $\cX$.
\end{remark}

The second issue arises from the equilibrium point itself and the points in its neighborhood, for the same reason as the infeasibility of constraint~\eqref{cstr:decrease} for negative slack variables. This is where Assumption~\ref{assumption:neighbourhood_A} comes in place: the existence of a neighborhood $\cA$ of the equilibrium point, understood as a target region, allows relaxing~\eqref{cstr:decrease} for points in $\cA$.

The third issue stems from numerical sensitivity of polyhedral intersection in the construction of $F_{\mathcal{D}_{N_d}}$. This propagates when computing the sets $\cF_v(v_u)$ at the tessellation vertices $\{v_u\}_{u\in [N_v]}$, potentially leading to inaccuracies that affect the feasibility of \eqref{optimisation_problem}. To address this, $2n$ auxiliary local \ac{LP}s are solved at each vertex $v_u$ to compute the bounds defining $\cF_v(v_u)$ accurately.

We now introduce the main algorithm implementing the proposed framework. The iterative procedure increases in tesselation density in direct relation to the selection of the new data points. Due to space limitations, we do not detail here the specific subroutines handling the double representation of polyhedral sets. These can employ off-the-shelf techniques or, in turn, employ optimization-based techniques to compute extreme realizations of the uncertainties as mentioned above. Also, further numerical refinements may also be considered regarding the tessellation size and its relation to the Lipschitz constants.

\begin{algorithm}[hbtp!]
    \caption{}\label{algorithm}
    \textbf{Input:} $\cX$, $\cA$, $M$, $\mu$\\
    \textbf{Output:} $\cL^{V}_{\alpha}$\\
    Generate initial data set $\cD_{N_d}$\\
    \While{the iteration limit is not reached}
        {Construct the uncertainty set $F_{\mathcal{D}_{N_d}}$\\
        Extract its vertices $\{v_u\}_{u\in [N_v]}$\\
        Compute $\cF_v(v_u)$ for all $u\in[N_v]$\\
        \eIf{$0\in\cF_v(v_u)$ for some $u\in[N_v]$}
            {Add data point at such vertices $v_u$}
            {Construct the tessellation $\{Y_j\}_{j\in [N_c]}$\\
            Solve the optimization problem~\eqref{optimisation_problem}\\
            Extract the highest sublevel set $\cL^{V}_{\alpha}$\\
            \eIf{all $s_u$ for $v_u\in\cL^{V}_{\alpha}$ are negative}
                {Terminate}
                {Add data points at vertices $v_u$ with $s_u\geq 0$}
            }
        }
\end{algorithm}

\subsection{Alternative PWA Lyapunov constructions}\label{SOCP_based}

The approach proposed in the previous subsection explicitly describes the uncertainty bounds before constructing the Lyapunov function, thereby decoupling the computational burden into two separate stages---uncertainty characterization and LP-based certification. As an alternative, in  \cite{tacchi2025,khattabi2025}, implicit bounds on the admissible values of $f(\cdot)$ are defined indirectly via Lyapunov decrease conditions. We recall here the convex optimization framework in~\cite{khattabi2025}, which builds on a data set $\{x_i,f_i\}_{i\in[N_d]}$ and a given tessellation $\{Y_j\}_{j\in[N_c]}$:
\begin{subequations}\label{eq:optprob}
\begin{align}
\min_{\{\gamma_{i,j},a_j,b_j,s_{k,j}\}} \quad
& \sum_{j} \sum_{k} s_{k,j} \label{eq:obj} \\
\text{s.t.} \quad
s_{k,j} &\geq -\mu
&& \forall j,\, \forall k, \label{eq:cnd1} \\
(a_j-a_{j'})^\top v_{k,j} &= b_{j'}-b_j
&& \forall j,j', \forall k \text{ s.t. } v_{k,j}\in Y_{j'}, \label{eq:cnd2}\\[-0.5ex]
\sum_{i} \gamma_{i,j} &= a_j
&& \forall j, \label{eq:cnd5} \\[-0.2ex]
\sum_{i} (\gamma_{i,j}^\top f_i + \| \gamma_{i,j}\|\,M\,\|v_{k,j} -x_i\| ) &\leq s_{k,j}
&& \forall j,\,\forall k, \label{eq:cnd6}
\end{align}
\end{subequations}
where $i\in[N_d]$ indexes the data points, $j,j'\in[N_c]$ index the cells $\{Y_j\}$, and $k\in[|Y_j|]$ indexes the vertices $\{v_{k,j}\}$ of the cell $Y_j$.

It is important to note that the parametrization of the \ac{PWA} Lyapunov function candidates is constrained through the auxiliary variables $\gamma_{i,j}$ in~\eqref{eq:cnd5}, which introduce additional optimization variables. As a result, the problem~\eqref{eq:optprob} becomes a second-order cone program (SOCP), with  its compactness coming at the expense of higher computational complexity due to the nonlinear constraints~\eqref{eq:cnd6}. 

In terms of comparison, the two approaches can be summarized by the following features:
\begin{itemize}
    \item A two-stage (separable and thus interpretable) certification for the LP-based construction in Section \ref{LP_based};
    \item a one-shot PWA certificate in Section
    \ref{SOCP_based} based on a monolithic SOCP formulation.  
\end{itemize}
As commonalities, both approaches depend on the choice of tessellation and can be iteratively updated by enriching the seed of data points. The interested reader is referred to \cite{khattabi2025} for convergence results and computational complexity of the SOCP approach. In the remainder of the paper, we focus on illustrating the LP-based approach.

\section{Numerical examples}\label{sec:numerical_examples}

In this section, we apply Algorithm~\ref{algorithm} to two numerical examples: a simple damped pendulum and an inverted-time Van der Pol oscillator. To avoid lifting issues and simplify the enforcement of constraints, triangular cells are adopted, and the tessellation is obtained via Delaunay triangulation. For the chosen initial data set, we take $N_d = 3$ with two random data points, the third being the equilibrium in the origin. The  simulations\footnote{Code is available at: \hyperlink{https://github.com/OumaymaK/LP_PWA_Lyapunov.git}{\texttt{github.com/OumaymaK/LP\_PWA\_Lyapunov}}} were performed on a laptop with an AMD Ryzen 7 7735U CPU and 32GB of RAM; MATLAB-based Multi-Parametric Toolbox  (MPT3)~\cite{MPT3},  YALMIP~\cite{lofberg2004}, and MOSEK~11.1~\cite{mosek} were used for the polyhedral operations, problem formulation, and optimization, respectively.

\subsection{Simple damped pendulum}

We consider the following system dynamics corresponding to a damped pendulum (see Fig.~\ref{fig:damped_pendulum_dynamics}):
\begin{equation*}
    \dot{x} = f(x) = \begin{pmatrix} x_2 \\ -\sin(x_1) - 2x_2 \end{pmatrix},
    \label{eq:pendulum}
\end{equation*}
with $x\in\cX = [-1,1]^2$, $\cA = [-0.1,0.1]^2$, and an upper bound on the Lipschitz constant $M = [1.15,3.15]^\mathsf{T}$.
\begin{figure}[htbp!]
    \centering
    \includegraphics[width = .65\linewidth]{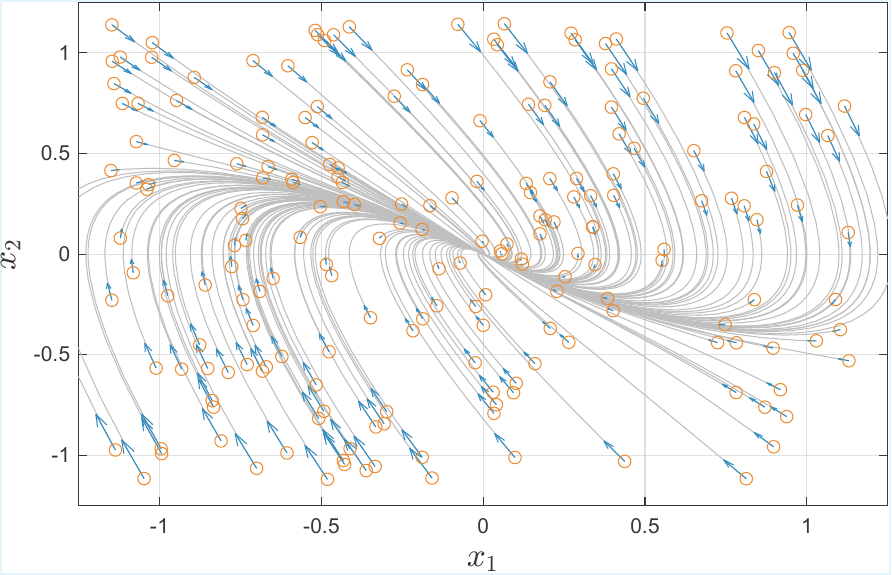}
    \caption{Dynamics of a simple damped pendulum.}
    \label{fig:damped_pendulum_dynamics}
\end{figure}

The resulting Lyapunov candidate, shown in Figure~\ref{fig:damped_pendulum_Lyapunov}, was obtained in 4 iterations with a total computation time of 472$\,$s. The approximation of \ac{RoA} is delimited by the red boundary and corresponds to the sublevel set \(\cL^V_{217.08}\) of the Lyapunov candidate, representing 49.81\% of the real \ac{RoA}. The contour curves match the shape of the true \ac{RoA} of the system dynamics. The dashed black lines depict the boundary of $\cX$. 
\begin{figure}[htbp!]
    \centering
    \includegraphics[width = .7\linewidth]{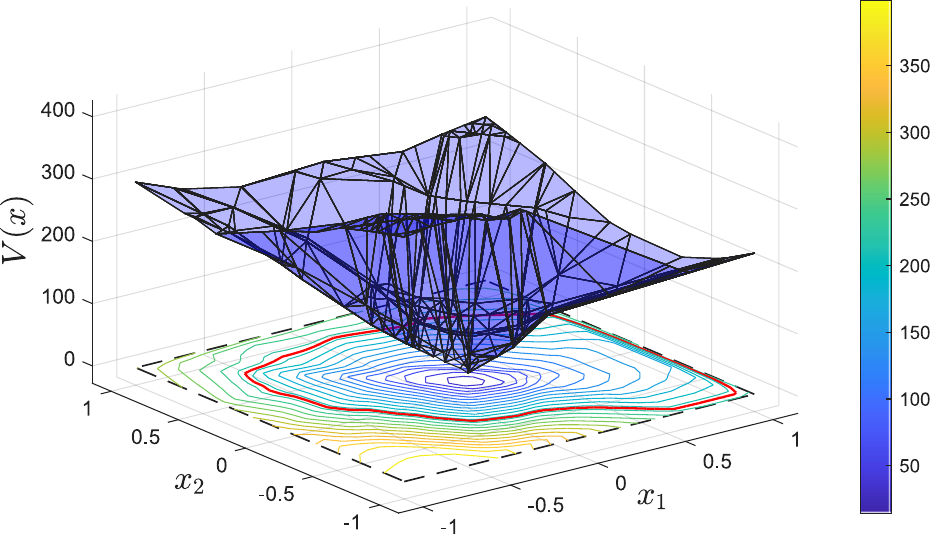}
    \caption{Lyapunov candidate and RoA for the damped pendulum.}
    \label{fig:damped_pendulum_Lyapunov}
\end{figure}

For a clearer analysis of the algorithm, selected performance metrics are reported in Table~\ref{tab:damped_pendulum}. Here, $N_d$ is the number of data points, $N_c$ the number of tessellation cells, $N_v$ the number of vertices, $T_\mathrm{data}$ the data processing time, $T_\mathrm{con}$ the constraint construction time, and $T_\mathrm{opt}$ the optimization time.
\begin{table}[htbp!]
    \centering
    \caption{Performance metrics for the damped pendulum example.}
    \begin{tabular}{ccccccc}
        \toprule
        Iteration & $N_d$ & $N_c$ & $N_v$ & $T_\mathrm{data}\,$[s] & $T_\mathrm{con}\,$[s] & $T_\mathrm{opt}\,$[s] \\
        \midrule
        1 & 3 & 78 & 41 & 19.3 & -- & -- \\
        2 & 33 & 600 & 262 & 122.1 & 11.5 & 2.1 \\
        3 & 87 & 619 & 271 & 128.7 & 11.0 & 2.0 \\
        4 & 147 & 735 & 321 & 158.1 & 15.1 & 2.1 \\
        \bottomrule
    \end{tabular}
    \label{tab:damped_pendulum}
\end{table}

According to Table~\ref{tab:damped_pendulum}, the optimization is first called at iteration 2 (after satisfying the necessary condition on the informativity of the data) and  ensures negative slack variables over the entire set by iteration 4. The number of data points is capped at 147. The most time-consuming step is data processing to extract $\cF_v(v_u)$ ($2n$ auxiliary \ac{LP}s are solved for each $v_u$), while the time to solve the LP~\eqref{optimisation_problem} is negligible.

\subsection{Inverted-time Van der Pol oscillator}

The inverted-time Van der Pol oscillator under study is described by the system dynamics (see Fig.~\ref{fig:Van_der_Pol_dynamics}):
\begin{equation*}
    \dot{x} = f(x) = \begin{pmatrix} -2x_2 \\ -10x_2(0.21-x_1^2) + 0.25x_2^2 \end{pmatrix},
    \label{eq:Van_der_Pol}
\end{equation*}
with $x\in\cX = [-0.5,0.5]^2$, $\cA = [-0.05,0.05]^2$, and an upper bound on the Lipschitz constant $M = [2.15,6.35]^\mathsf{T}$.
\begin{figure}[htbp!]
    \centering
    \includegraphics[width = .65\linewidth]{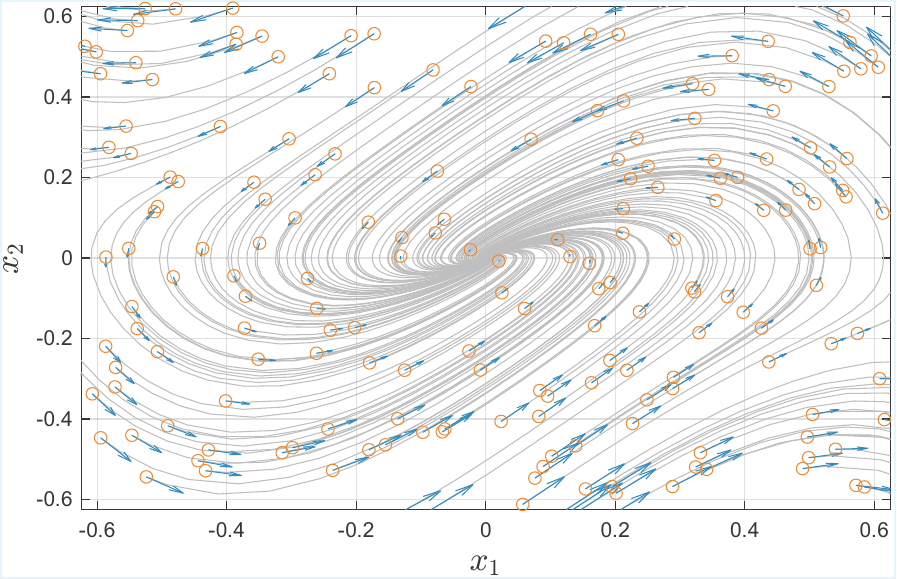}
    \caption{Dynamics of an inverted-time Van der Pol oscillator.}
    \label{fig:Van_der_Pol_dynamics}
\end{figure}

The simulation took 937.4$\,$s to converge. Figure~\ref{fig:Van_der_Pol_Lyapunov} shows the obtained Lyapunov candidate. In this case, the sublevel set $\cL^V_{63.16}$, depicted in red, is one connected set containing the equilibrium, and it represents 45.53\% of the true \ac{RoA}.
\begin{figure}[htbp!]
    \centering
    \includegraphics[width = .7\linewidth]{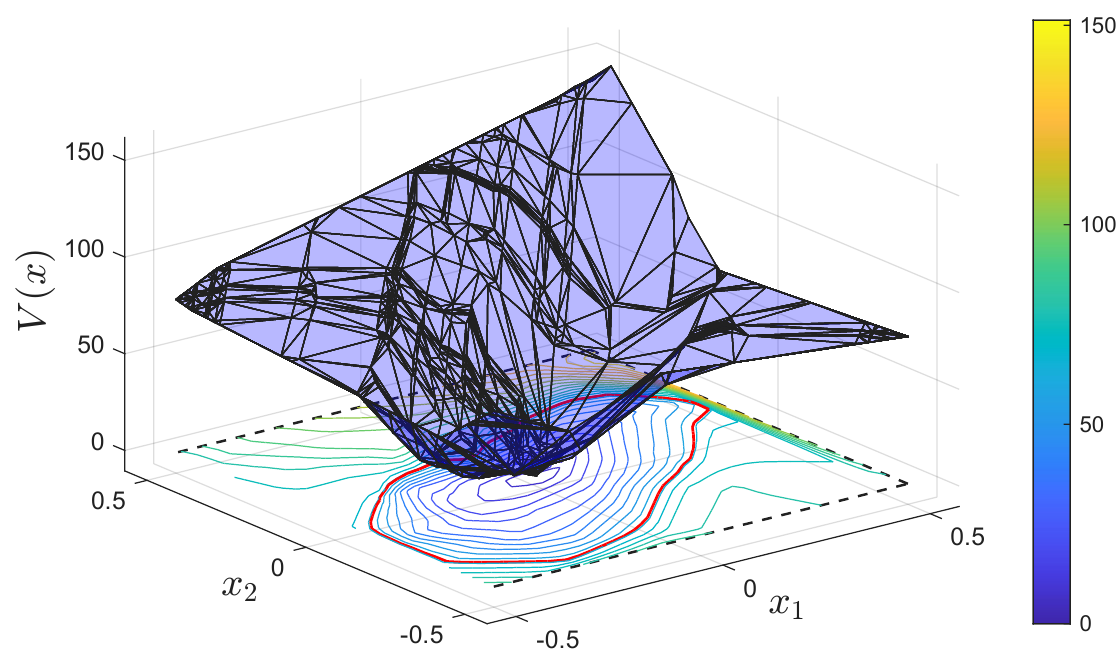}
    \caption{Lyapunov candidate and RoA for the Van der Pol oscillator.}
    \label{fig:Van_der_Pol_Lyapunov}
\end{figure}

As seen in Table~\ref{tab:Van_der_Pol}, only three iterations were required to obtain a result for the inverse-time Van der Pol oscillator, albeit with more time and data. However, non-negative slack variables persist outside the certified \ac{RoA}, indicating that the Lyapunov decrease condition is not satisfied globally, in contrast to the previous example, where all slack variables are negative over the entire domain.
\begin{table}[htbp!]
    \centering
    \caption{Performance metrics for the Van der Pol example.}
    \begin{tabular}{ccccccc}
        \toprule
        Iteration & $N_d$ & $N_c$ & $N_v$ & $T_\mathrm{data}\,$[s] & $T_\mathrm{con}\,$[s] & $T_\mathrm{opt}\,$[s] \\
        \midrule
        1 & 3 & 104 & 57 & 25.4 & -- & -- \\
        2 & 41 & 1560 & 788 & 296.5 & 33.4 & 7.4\\
        3 & 370 & 1684 & 852 & 303.3 & 35.8 & 6.9\\
        \bottomrule
    \end{tabular}
    \label{tab:Van_der_Pol}
\end{table}

\section{Conclusion}\label{sec:conclusion}
This paper proposes a data-driven framework for constructing PWA Lyapunov functions and approximating the region of attraction for nonlinear systems with unknown dynamics. The approach relies on the explicit construction of a polyhedral uncertainty set from data, which induces a differential inclusion describing all admissible dynamics. This representation enables a tractable LP formulation for enforcing Lyapunov conditions over the uncertainty set.

A key feature of the proposed method is the separation of uncertainty characterization from Lyapunov certification, thereby improving interpretability and enabling iterative refinement via data enrichment. The resulting PWA Lyapunov candidates provide approximated regions of attraction, even when the available data is sparse. The comparison with an alternative SOCP formulation highlights the trade-off between computational complexity and structural transparency.

The approach nevertheless has limitations. There is a trade-off between the conservativeness in the choice of the Lipschitz constant and the number of data points necessary for convergence, which affects the data processing part, and the scalability of the method in consequence. Future work includes improving scalability and refining the state-space tessellation to derive sufficient conditions for its construction.

\bibliographystyle{ieeetr}
\bibliography{RefCDC26.bib}

@article{bramburger2024,
  title = {Auxiliary Functions as {K}oopman Observables: {D}ata-Driven Analysis of Dynamical Systems via Polynomial Optimization},
  author = {Bramburger, Jason J. and Fantuzzi, Giovanni},
  year = 2024,
  month = feb,
  journal = {J. Nonlinear Sci.},
  volume = {34},
  number = {1}
}

@article{brayton1980,
  title = {Constructive Stability and Asymptotic Stability of Dynamical Systems},
  author = {Brayton, R. and Tong, C.},
  year = 1980,
  journal = {IEEE Trans. Circuits Syst.},
  volume = {27},
  number = {11},
  pages = {1121--1130}
}

@article{brayton1979,
  title = {\hspace{-.1em}{S}tability\hspace{-.15em} of\hspace{-.15em} Dynamical\hspace{-.15em} Systems:\hspace{-.15em} {{A}}\hspace{-.15em} Constructive\hspace{-.1em} Approach\hspace{-.1em}},
  author = {Brayton, R. and Tong, C.},
  year = 1979,
  journal = {IEEE\hspace{-.2em} Trans.\hspace{-.2em} Circuits\hspace{-.2em} Syst.\hspace{-.1em}},
  volume = {26},
  number = {4},
  pages = {224--234}
}

@article{grune2021,
author = {Lars Grüne},
title = {Computing {L}yapunov functions using deep neural networks},
journal = {J. Comput. Dyn.},
volume = {8},
number = {2},
pages = {131--152},
year = {2021}
}

@inproceedings{MPT3,
  title = {Multi-Parametric Toolbox 3.0},
  booktitle = {Proc. ECC},
  author = {Herceg, Martin and Kvasnica, Michal and Jones, Colin N. and Morari, Manfred},
  year = 2013,
  pages = {502--510}
}

@article{hmamed1994,
  title = {Comments on "{{Vector}} Norms as {{Lyapunov}} Functions for Linear Systems" [with Reply]},
  author = {Hmamed, A. and Kiendl, H. and Adamy, J. and Stelzner, P.},
  year = 1994,
  journal = {IEEE Trans. Autom. Control},
  volume = {39},
  number = {12},
  pages = {2522--2523}
}

@article{su1994,
  title = {Comments on "{{Stability}} Margin Evaluation for Uncertain Linear Systems"},
  author = {Su, Juing-Huei},
  year = 1994,
  journal = {IEEE Trans. Autom. Control},
  volume = {39},
  number = {12},
  pages = {2523--2524}
}

@article{julian1999,
  title = {A Parametrization of Piecewise Linear {{Lyapunov}} Functions via Linear Programming},
  author = {Julian, Pedro and Guivant, Jose and Desages, Alfredo},
  year = 1999,
  journal = {Int. J. Control},
  volume = {72},
  number = {7-8},
  pages = {702--715}
}

@article{kiendl1992,
  title = {Vector Norms as {{Lyapunov}} Functions for Linear Systems},
  author = {Kiendl, H. and Adamy, J. and Stelzner, P.},
  year = 1992,
  journal = {IEEE Trans. Autom. Control},
  volume = {37},
  number = {6},
  pages = {839--842}
}

@inproceedings{kim2024,
  title = {Estimation of constraint admissible invariant set with neural {L}yapunov function},
  booktitle = {Proc. IEEE CDC},
  author = {Kim, Dabin and Kim, H. Jin},
  year = 2024,
  pages = {5032--5039},
  address = {Milan, Italy}
}

@article{kolter2019,
  title={Learning stable deep dynamics models},
  author={Kolter, J Zico and Manek, Gaurav},
  journal={Adv. Neural Inf. Process. Syst.},
  year={2019}
}

@article{loskot1998,
  title = {Further comments on "{{Vector}} norms as {{Lyapunov}} functions for linear systems"},
  author = {Loskot, K. and Polanski, A. and Rudnicki, R.},
  year = 1998,
  journal = {IEEE Trans. Autom. Control},
  volume = {43},
  number = {2},
  pages = {289--291}
}

@article{lyapunov1992,
  title = {The General Problem of the Stability of Motion},
  author = {Lyapunov, Aleksandr Mikhailovich},
  year = 1992,
  journal = {Int. J. Control},
  volume = {55},
  number = {3},
  pages = {531--534}
}

@article{martin2024,
  title = {Data-Driven System Analysis of Nonlinear Systems Using Polynomial Approximation},
  author = {Martin, Tim and Allg{\"o}wer, Frank},
  year = 2024,
  journal = {IEEE Trans. Autom. Control},
  volume = {69},
  number = {7},
  pages = {4261--4274}
}

@article{miller2024,
  title = {Data-Driven Superstabilizing Control Under Quadratically-Bounded Errors-in-Variables Noise},
  author = {Miller, Jared and Dai, Tianyu and Sznaier, Mario},
  year = 2024,
  journal = {IEEE Control Syst. Lett.},
  volume = {8},
  pages = {1655--1660}
}

@misc{mosek,
  title = {The {{MOSEK}} Optimization Toolbox for {{MATLAB}} Manual. {{Version}} 11.1.10.},
  author = {MOSEK ApS},
  year = 2026
}

@misc{nagumo1942,
  title = {{\"U}ber Die {{Lage}} Der {{Integralkurven}} Gew\"ohnlicher {{Differentialgleichungen}}},
  author = {Nagumo, Mitio},
  year = 1942
}

@article{ohta1993,
  title = {Computer Generated {{Lyapunov}} Functions for a Class of Nonlinear Systems},
  author = {Ohta, Y. and Imanishi, H. and Gong, L. and Haneda, H.},
  year = 1993,
  journal = {IEEE Trans. Circuits Syst. I},
  volume = {40},
  number = {5},
  pages = {343--354}
}

@article{tacchi2025,
  title = {Robustly Learning Regions of Attraction From Fixed Data},
  author = {Tacchi, Matteo and Lian, Yingzhao and Jones, Colin N.},
  year = 2025,
  journal = {IEEE Trans. Autom. Control},
  volume = {70},
  number = {3},
  pages = {1576--1591}
}

@article{wang2024,
  title = {{A}ctor-Critic\hspace{-.2em} Physics-Informed\hspace{-.2em} Neural\hspace{-.2em} {L}yapunov\hspace{-.2em} Control\hspace{-.1em}},
  author = {Wang, Jiarui and Fazlyab, Mahyar},
  year = 2024,
  journal = {IEEE Control Syst. Lett.},
  volume = {8},
  pages = {1751--1756}
}

@article{zheng2024,
  title = {Data-Driven Safe Control of Discrete-Time Non-Linear Systems},
  author = {Zheng, Jian and Miller, Jared and Sznaier, Mario},
  year = 2024,
  journal = {IEEE Control Syst. Lett.},
  volume = {8},
  pages = {1553--1558}
}

@inproceedings{lofberg2004,
    address = {Taipei, Taiwan},
    author = {L{\"{o}}fberg, J.},
    booktitle = {Proc. CACSD},
    title = {{YALMIP}: {A} Toolbox for Modeling and Optimization in {MATLAB}},
    year = {2004}
}

@article{gong1994,
  title={Stability margin evaluation for uncertain linear systems},
  author={Gong, C and Thompson, S},
  journal={IEEE Trans. Autom. Control},
  volume={39},
  number={3},
  pages={548--550},
  year={1994}
}

@article{kamalapurkar2020,
  title = {On reduction of differential inclusions and {L}yapunov stability},
  author = {Kamalapurkar, Rushikesh and Dixon, Warren E and Teel, Andrew R},
  journal = {ESAIM Control Op. Ca. Va.},
  volume = {26},
  pages = {24},
  year = {2020}
}

@article{dacs2023,
  title = {Robust control barrier functions with uncertainty estimation},
  author = {Da{\c{s}}, Ersin and Wei, Skylar X and Burdick, Joel W},
  journal = {arXiv:2304.08538},
  year = {2023}
}

@inproceedings{khattabi2025,
  title = {Sequentially learning regions of attraction from data},
  author = {Khattabi, Oumayma and Tacchi, Matteo and Olaru, Sorin},
  booktitle = {Proc. IEEE Mediterr. Control Autom.},
  pages = {589--594},
  year = {2025}
}

@inproceedings{jin2020data,
  title={Data-driven model invalidation for unknown Lipschitz continuous systems via abstraction},
  author={Jin, Zeyuan and Khajenejad, Mohammad and Yong, Sze Zheng},
  booktitle={Proc. ACC},
  pages={2975--2980},
  year={2020},
  organization={IEEE}
}

@article{makdesi2021efficient,
  title={Efficient data-driven abstraction of monotone systems with disturbances},
  author={Makdesi, Anas and Girard, Antoine and Fribourg, Laurent},
  journal={IFAC-PapersOnLine},
  volume={54},
  number={5},
  pages={49--54},
  year={2021},
  publisher={Elsevier}
}

\end{document}